\documentclass[preprint]{imsart}

\RequirePackage{amsthm,amsmath,amsfonts,amssymb}
\RequirePackage[numbers,sort&compress]{natbib}
\RequirePackage[colorlinks,citecolor=blue,urlcolor=blue]{hyperref}
\RequirePackage{graphicx}
\usepackage{xcolor}

\newcommand{\PP}{\mathbb{P}}
\newcommand{\EE}{\mathbb{E}}

\startlocaldefs
\theoremstyle{plain}

\newtheorem{theorem}{Theorem}[section]
\newtheorem{proposition}[theorem]{Proposition}
\newtheorem{corollary}[theorem]{Corollary}
\newtheorem{lemma}[theorem]{Lemma}
\newtheorem{remark}[theorem]{Remark}
\theoremstyle{definition}

\endlocaldefs

\begin{document}

\begin{frontmatter}
\title{The Frog Model on $\mathbb{Z}$ with Discrete Weibull Lifetimes and Random Parameter $p$}
\runtitle{Frog Model with Discrete Weibull Lifetimes and Random $p$}


\begin{aug}
\author[A]{\fnms{J. Hermenegildo}~\snm{R. González} \thanks{[\textbf{Corresponding author indication should be put in the Acknowledgment section if necessary.}]}\ead[label=e1]{hermenegildo.ramirez@usp.br}\orcid{0000-0001-5610-4712}},
\author[A]{\fnms{Gustavo}~\snm{O. Carvalho}\ead[label=e2]{gustavoodc@ime.usp.br}\orcid{0000-0002-5047-5544}}
\and
\author[A]{\fnms{Fábio}~\snm{P. Machado}\ead[label=e3]{fmachado@ime.usp.br}\orcid{0000-0003-3053-4372}}

\address[A]{Universidade de São Paulo\printead[presep={,\ }]{e1,e2,e3}}
\end{aug}

\begin{abstract}
We study the frog model on $\mathbb{Z}$ with particle–wise \emph{discrete Weibull} lifetimes. Each particle has an i.i.d.\ survival parameter $\pi\in(0,1)$; conditionally on $\pi=p$, its lifetime $\Xi$ satisfies
\[
\PP(\Xi\ge k\mid \pi=p)=p^{k^{\gamma}},\qquad k\in\mathbb{N}_0,\ \gamma>0.
\]
The law of $\pi$ has right–edge density
\[
f_\pi(u)\ \sim\ (1-u)^{\beta-1}\,L\!\big((1-u)^{-1}\big)\qquad (u\uparrow 1),
\]
with $\beta>0$ and $L$ slowly varying; let $\eta$ denote the common law of the i.i.d.\ initial occupation numbers $\{\eta_x\}_{x\in\mathbb{Z}}$. The survival–parameter distribution strictly extends the Beta family, while the lifetime distribution extends the geometric case.  
We prove a sharp extinction–survival dichotomy with the $\gamma$–dependent threshold
\[
\beta_c:=\frac{1}{2\gamma}.
\]
If $\beta>\beta_c$ and $E(\eta)<\infty$, the process becomes extinct almost surely; if $\beta<\beta_c$ and $\PP(\eta=0)<1$, it survives with positive probability. At the boundary $\beta=\beta_c$ we provide explicit criteria in terms of $\limsup/\liminf$ of $L(n^{2\gamma})$. The case $\gamma=1$ (geometric lifetimes) recovers the benchmark $\beta_c=\tfrac12$ and the critical refinements previously obtained for random geometric lifetimes.

\end{abstract}



\end{frontmatter}

\section{Introduction and results}

Interacting particle systems that combine random motion with local activation rules—most prominently the \emph{frog model}—offer a parsimonious framework to probe survival/extinction phenomena and front propagation in disordered media. Since its introduction by Alves, Machado and Popov \cite{Alves2002}, the model has revealed genuine phase transitions on trees and higher–dimensional lattices \cite{Fontes2004,Lebensztayn2005,Gallo2023}, with sharp or near–sharp criteria depending on geometry and initial densities \cite{Lebensztayn2005,Lebensztayn2019,Lebensztayn2020}. On $\mathbb{Z}$, however, geometric lifetimes with a \emph{fixed} survival parameter $p$ often entail almost–sure extinction under mild assumptions, unless additional structure (e.g.\ drift or spatial heterogeneity) is introduced \cite{Bertacchi2014,Gantert2009}. A key step forward was obtained by Carvalho and Machado \cite{CarvalhoMachado2025}, who considered i.i.d.\ \emph{random survival parameters} $\pi\in(0,1)$ and, for geometric lifetimes, identified a sharp \emph{edge} threshold at $\beta=\tfrac12$ driven by the mass of $\pi$ near $1$. A subsequent paper extends this threshold to broad classes of right–edge densities with slowly varying corrections; see \cite{JGF}.

In the present paper we keep the same one--dimensional frog model with random survival parameter $\pi$, but we replace geometric lifetimes by discrete Weibull lifetimes. Here $\Xi$ denotes the lifetime of each active frog, i.e., the number of steps it performs after activation before it dies. Conditionally on $\pi=p$, we assume
\[
\mathbb{P}(\Xi \ge k \mid \pi=p)=p^{k^\gamma}, \qquad k\in\mathbb{N},\ \gamma>0.
\]
Under a standard regularity assumption on the distribution of $\pi$ near its right endpoint $p=1$, we obtain a sharp extinction--survival transition at the critical exponent $\beta_c=1/(2\gamma)$: extinction holds for $\beta>\beta_c$, survival holds for $\beta<\beta_c$, and the boundary case $\beta=\beta_c$ depends on the slowly varying factor. In particular, when $\gamma=1$ we recover the geometric case and the known threshold $\beta_c=\tfrac12$.

\medskip
\subsection*{Motivation: from geometric to discrete Weibull lifetimes}

\emph{Why move beyond geometric lifetimes?} Geometric lifetimes are memoryless: the discrete hazard \(h_k:=\PP(\Xi=k\,|\,\Xi\ge k)\) is constant in~$k$, so past survival contains no information about future risk. In many contexts—reliability, epidemiological clearance, device aging—risk is age–dependent, exhibiting either \emph{burn–in} (decreasing hazard after early failures) or \emph{wear–out} (increasing hazard with age); see the classical reliability framework of Barlow–Proschan \cite{BarlowProschan1975}.

To capture both behaviors in discrete time with a single, tractable family, we adopt the type–I \emph{discrete Weibull} distribution \cite{NakagawaOsaki1975} with paremeter \(\gamma>0\), whose survival function is given by 
\begin{equation}\label{eq:dw-survival-intro}
S(k)\;=\;\PP(\Xi\ge k\,|\,\pi=p)\;=\;p^{\,k^{\gamma}},\qquad k\in\mathbb{N}_0,\ \ 0<p<1,\ \ \gamma>0,
\end{equation}
which is the integer–time analogue of the continuous Weibull tail $e^{-\lambda t^{\gamma}}$ (via $p=e^{-\lambda}$). The shape parameter $\gamma$ is the single knob for aging:
\[
h_k\ =\ 1-\frac{S(k+1)}{S(k)}\ =\ 1-p^{(k+1)^\gamma-k^\gamma},
\]
so the hazard is increasing when $\gamma>1$ (wear–out), constant when $\gamma=1$ (memoryless, i.e., geometric), and decreasing when $0<\gamma<1$ (burn–in). Thus \eqref{eq:dw-survival-intro} preserves analytical simplicity while aligning the model with the two canonical hazard regimes.

Figure~\ref{fig:weibull-tail} depicts the discrete Weibull tail $k\mapsto p^{k^{\gamma}}$ for representative $(p,\gamma)$: for $\gamma>1$ long lifetimes are strongly suppressed relative to the geometric case, while for $0<\gamma<1$ the tail is heavier. 

\begin{figure}[h!]
  \centering
  \makebox[\textwidth][c]{%
    \includegraphics[width=1.2\linewidth,height=0.4\textheight,keepaspectratio=false]{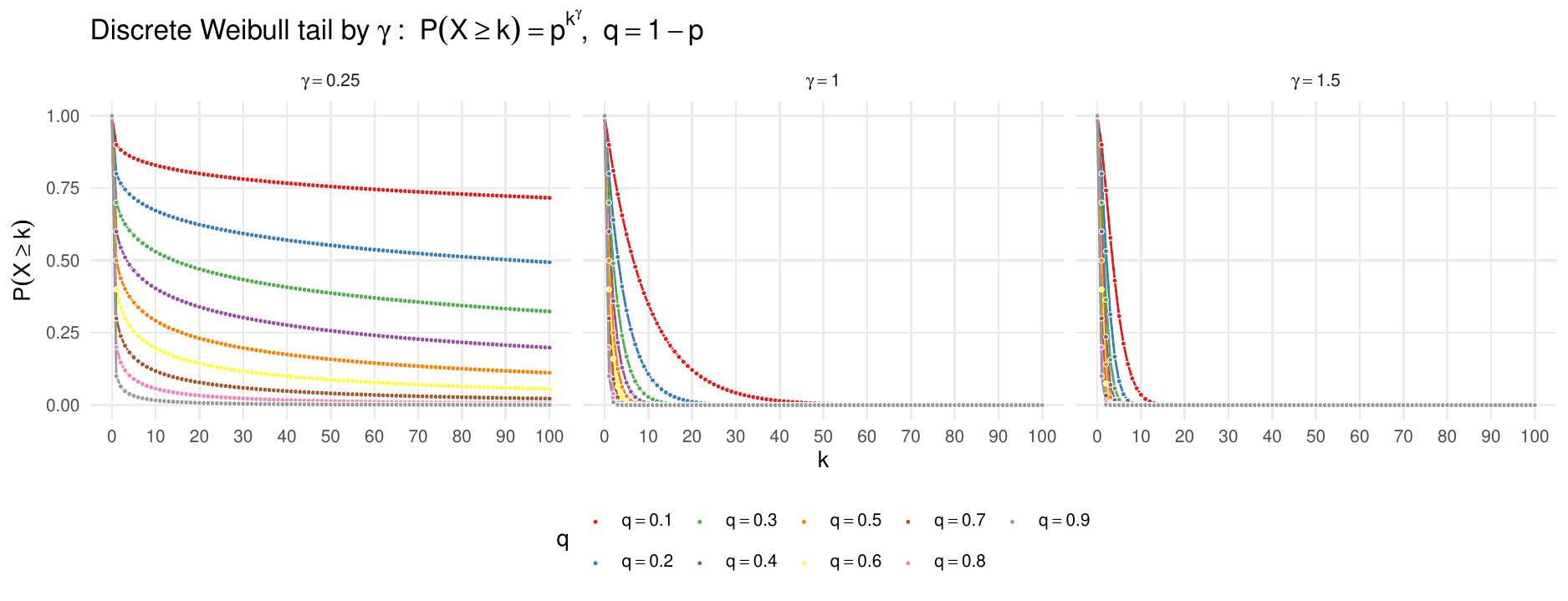}%
  }
  \caption{Survival tails of the discrete Weibull, $\PP(\Xi\ge k)=p^{k^{\gamma}}$, for several values of $p$ and $\gamma$. For $\gamma>1$ the tail decays faster than geometric (wear--out); $\gamma=1$ recovers the geometric (memoryless) case; and $0<\gamma<1$ yields heavier tails (burn--in).}
  \label{fig:weibull-tail}
\end{figure}

As in \cite{JGF}, we quantify the availability of almost–immortal particles via the asymptotic behavior
\begin{equation}\label{eq:edge-behavior}
f_\pi(u)\ \sim\ (1-u)^{\beta-1}\,L\!\Big(\frac{1}{1-u}\Big)\qquad (u\uparrow 1),
\end{equation}
with $\beta>0$ and $L$ slowly varying at $\infty$. This class contains the Beta family and heterogeneous right edges modulated by $L$. An extensive list of examples satisfying \eqref{eq:edge-behavior} is given in \cite{JGF}. Intuitively, a smaller $\beta$ corresponds to a thicker right edge and hence to more long–lived particles.

The choice \eqref{eq:dw-survival-intro} reconciles discrete-time dynamics with the two canonical aging regimes while keeping the model analytically tractable. For $\gamma>1$ the tail $k\mapsto p^{k^\gamma}$ decays faster than the geometric tail (reflecting wear–out), for $0<\gamma<1$ it decays more slowly in a stretched–exponential fashion (reflecting burn–in), and $\gamma=1$ recovers the memoryless geometric case. When, in addition, the survival parameter has a density satisfying the right–edge condition \eqref{eq:edge-behavior}, the mass near $\pi\simeq1$—encoded by $(\beta,L)$—interacts with the random-walk hitting-time distribution in a way made precise by our one-particle estimates (Theorems~\ref{thm:gamma>1-final-EN-eps}–\ref{thm:gamma<1-final-EN-eps}). As a consequence, the extinction–survival criterion exhibits the $\gamma$–dependent critical curve
\[
\beta_c\;:=\;\frac{1}{2\gamma}.
\]
In particular, for $E(\eta)<\infty$ there is extinction when $\beta>\beta_c$, and for $\PP(\eta=0)<1$ there is survival with positive probability when $\beta<\beta_c$, with critical envelope conditions spelled out at $\beta=\beta_c$. The case $\gamma=1$ recovers the geometric benchmark $\beta_c=\tfrac12$ from \cite{CarvalhoMachado2025} and its regularly varying extensions in \cite{JGF}.

Figure~\ref{fig:diagrama_Weibull} plots the critical curve $\gamma_c(\beta)=1/(2\beta)$ in the $(\beta,\gamma)$–plane, separating almost–sure extinction from survival with positive probability, as formalized in Theorem~\ref{thm:general-gamma}.

\begin{figure}[h!]
  \centering
  \makebox[\textwidth][c]{%
    \includegraphics[width=1.2\linewidth,height=0.4\textheight,keepaspectratio=false]{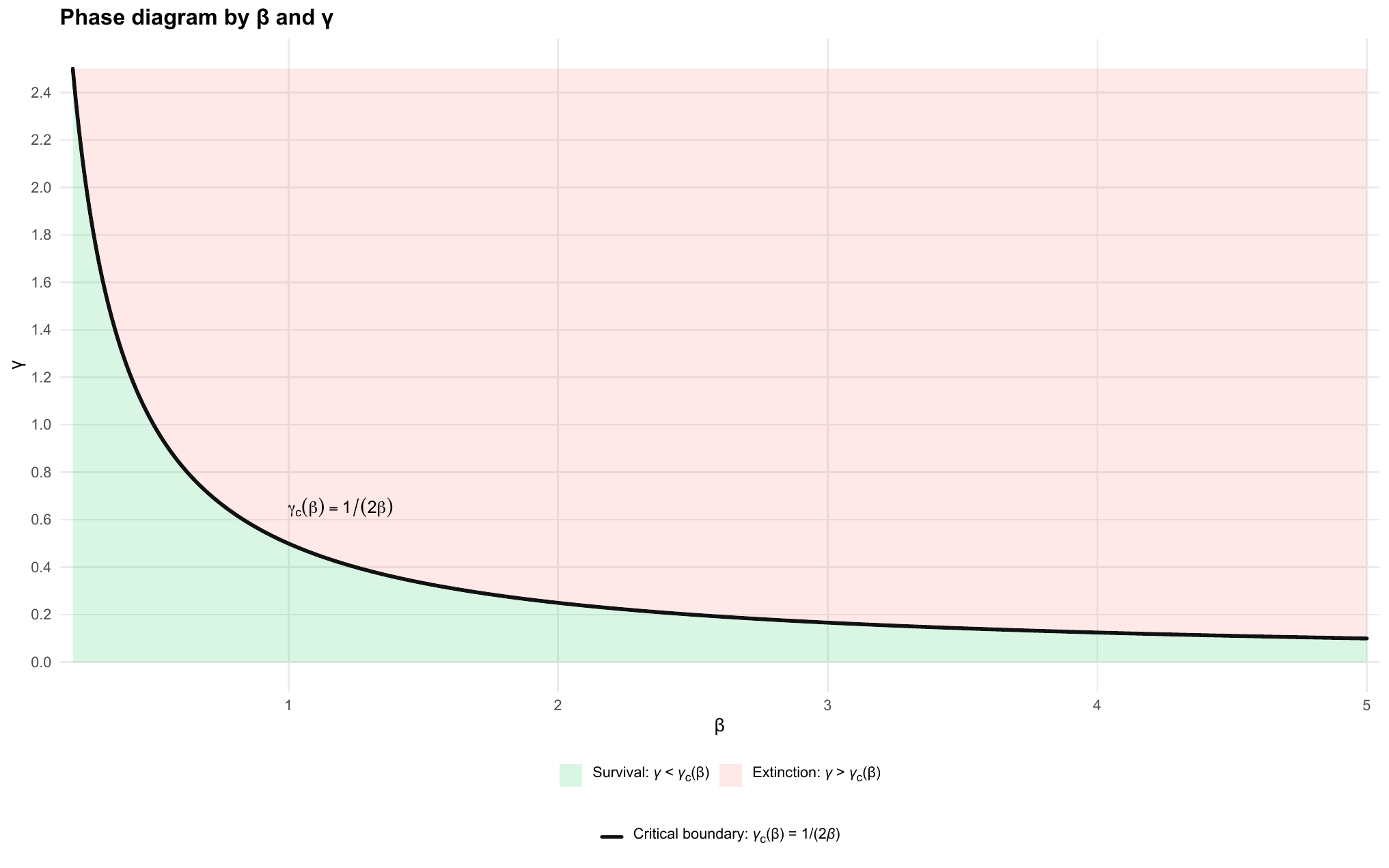}%
  }
  \caption{Phase diagram in the $(\beta,\gamma)$ plane. The critical curve $\gamma_c(\beta)=1/(2\beta)$ separates survival with positive probability (below) from extinction.}
  \label{fig:diagrama_Weibull}
\end{figure}

\subsection*{Model and notation on $\mathbb{Z}$}
Let $\mathbb{N}=\{1,2,3,\dots\}$ and $\mathbb{N}_0=\mathbb{N}\cup\{0\}$. For each $x\in\mathbb{Z}$, let $\eta_x\in\mathbb{N}_0$ be the initial number of sleeping particles at $x$. Particles are indexed by $(x,i)$ with $1\le i\le \eta_x$. The families
\[
\{\eta_x\}_{x\in\mathbb{Z}},\quad \{\pi_{x,i}\}_{x\in\mathbb{Z},\,i\in\mathbb{N}},\quad
\{(S^{x,i}_n)_{n\in\mathbb{N}_0}\}_{x\in\mathbb{Z},\,i\in\mathbb{N}}
\]
are mutually independent, where $(S^{x,i}_n)$ is a simple symmetric random walk with $S^{x,i}_0=x$, and $\pi_{x,i}\in(0,1)$ are i.i.d.\ with law $\pi$ (density $f_\pi$ when it exists). Given $\pi_{x,i}=p$ and $\gamma>0$, the lifetime $\Xi_{x,i}\in\mathbb{N}_0$ has \emph{discrete Weibull tail}
\begin{equation}\label{eq:dw-survival-model}
\PP(\Xi_{x,i}\ge k\mid \pi_{x,i}=p)\;=\;p^{\,k^{\gamma}},\qquad k\in\mathbb{N}_0.
\end{equation}
Conditional on $\{\pi_{x,i}\}_{x\in\mathbb{Z},\,i\in\mathbb{N}}$, the lifetimes $\{\Xi_{x,i}\}_{x\in\mathbb{Z},\,i\in\mathbb{N}}$ are independent of \(
\{\eta_x\}_{x\in\mathbb{Z}}\) and \(\{(S^{x,i}_n)_{n\in\mathbb{N}_0}\}_{x\in\mathbb{Z},\,i\in\mathbb{N}}\) (and of each other).

At time $0$ all particles at the origin are \emph{active} and all others are \emph{inactive}. For every $(x,i)\in \mathbb{Z}\times\mathbb{N}$, the $i$-th particle of vertex $x\in\mathbb{Z}$ (considering $i\le \eta_x$), if activated at some instant $t_{x,i}$, is considered \emph{alive} at any instant $t\le t_{x,i}+\Xi_{x,i}$ and \emph{dead} at any $t>t_{x,i}+\Xi_{x,i}$. At every instant of time $t\in\mathbb{N}$, each active and alive particle takes one step of its simple symmetric random walk on $\mathbb{Z}$. Whenever an active particle visits a site, all particles
present there become active and evolve independently by the same rules. We denote this system by $\mathrm{FM}(\mathbb{Z},\pi,\eta,\gamma)$, where $\pi$ is the common law of the survival parameters, $\eta$ is the law of the initial number of particles per vertex, and $\gamma>0$ is the parameter of the discrete Weibull distribution.

\subsection*{Main results}

Our main result is related to the survival of the frog model. We say that the frog model \emph{survives} if, at every time instant $t>0$, there is at least one active and alive particle. Otherwise, we say that the model dies out.

We are now ready to state the main theorem of this paper.

\begin{theorem}[General extinction–survival dichotomy]\label{thm:general-gamma}
Assume \eqref{eq:edge-behavior} and that the survival-parameter law $\pi$ has density satisfying, as $u\uparrow 1$,
\[
f_\pi(u)\ \sim\ (1-u)^{\beta-1}\,L\!\Big(\frac{1}{1-u}\Big),
\]
with $\beta>0$, and $L$ slowly varying at $\infty$. Fix $\gamma>0$ and set $\beta_c=\tfrac{1}{2\gamma}$.

\medskip
\noindent\textbf{(A) Case $\gamma>1$.} Let \(c_0>0\). Define
\[
K^{\uparrow}_{\gamma,\beta}\ :=\ 2\gamma\ \Gamma(2\beta\gamma)\,2^{\beta(1-\gamma)}
\Bigg(\frac{\Gamma\!\bigl(1-\tfrac{1}{2\gamma}\bigr)}{\sqrt{\pi}}\Bigg)^{-2\gamma\beta},
\,
K^{\downarrow}_{\gamma,\beta}(c_0)\ :=\ 2\gamma\ \theta(c_0)\!\int_{0}^{\infty}\! y^{2\gamma\beta-1}e^{-c_0^{\gamma}y^{2\gamma}}\,dy,
\]
where $\theta(c_0):=\tfrac12\bigl(1-\Phi(1/\sqrt{c_0})\bigr)\in(0,1)$ and $\Phi$ is the standard Gaussian cdf.
Then:
\begin{itemize}
\item[(i) \emph{Extinction.}] If $E(\eta)<\infty$ and either $\beta>\beta_c$, or $\beta=\beta_c$ and
\[
2K^{\uparrow}_{\gamma,\beta}\ \limsup_{n\to\infty} L(n^{2\gamma})\ <\ \frac{1}{E(\eta)},
\]
then $\mathbb{P}\big[\mathrm{FM}(\mathbb{Z},\pi,\eta)\text{ survives}\big]=0$.
\item[(ii) \emph{Survival.}] If $\mathbb{P}(\eta=0)<1$, then
$\mathbb{P}\big[\mathrm{FM}(\mathbb{Z},\pi,\eta)\text{ survives}\big]>0$ in either of the following cases:
\begin{itemize}
\item $\beta<\beta_c$;
\item $\beta=\beta_c$ and, for some $c_0>0$,
\[
K^{\downarrow}_{\gamma,\beta}(c_0)\ \liminf_{n\to\infty} L(n^{2\gamma})\ >\ \frac{1}{E(\eta)}
\quad \textup{(with $1/E(\eta):=0$ when $E(\eta)=\infty$)}.
\]
\end{itemize}
\end{itemize}

\medskip
\noindent\textbf{(B) Case $\gamma<1$.} Define
\[
K^{\uparrow}_{\gamma,\beta}\ :=\ 2\gamma\ \Gamma(2\beta\gamma)\,2^{-\beta\gamma}\,
\frac{\Gamma(\beta)}{\gamma\,\Gamma(\beta\gamma)},
\qquad
K^{\downarrow}_{\gamma,\beta}\ :=\ 2\gamma\ \Gamma(2\beta\gamma)\,2^{-\beta(1+\gamma)}\,
\frac{\Gamma(\beta)}{\gamma\,\Gamma(\beta\gamma)}.
\]
Then:
\begin{itemize}
\item[(i) \emph{Extinction.}] If $E(\eta)<\infty$ and either $\beta>\beta_c$, or $\beta=\beta_c$ and
\[
2K^{\uparrow}_{\gamma,\beta}\ \limsup_{n\to\infty} L(n^{2\gamma})\ <\ \frac{1}{E(\eta)},
\]
then $\mathbb{P}\big[\mathrm{FM}(\mathbb{Z},\pi,\eta)\text{ survives}\big]=0$.
\item[(ii) \emph{Survival.}] If $\mathbb{P}(\eta=0)<1$, then
$\mathbb{P}\big[\mathrm{FM}(\mathbb{Z},\pi,\eta)\text{ survives}\big]>0$ in either of the following cases:
\begin{itemize}
\item $\beta<\beta_c$;
\item $\beta=\beta_c$ and
\[
K^{\downarrow}_{\gamma,\beta}\ \liminf_{n\to\infty} L(n^{2\gamma})\ >\ \frac{1}{E(\eta)}
\quad \textup{(with $1/E(\eta):=0$ when $E(\eta)=\infty$)}.
\]
\end{itemize}
\end{itemize}

\end{theorem}
The Corollary \ref{cor:general-gamma-envelope} follows immediately from the proofs of Theorems \ref{thm:gamma>1-final-EN-eps}, \ref{thm:gamma<1-final-EN-eps}, and \ref{thm:general-gamma}.

\begin{corollary}[One-sided envelope criteria]\label{cor:general-gamma-envelope}
Assume that there exist $U\in(0,1)$, $\beta>0$ and a function $L$ slowly varying at $+\infty$ such that, for all $u\ge U$,
either
\[
  f_\pi(u)\ \le\ (1-u)^{\beta-1}\,L\!\Big(\frac{1}{1-u}\Big)
  \qquad\textup{(upper envelope)}
\]
or
\[
  f_\pi(u)\ \ge\ (1-u)^{\beta-1}\,L\!\Big(\frac{1}{1-u}\Big)
  \qquad\textup{(lower envelope)}.
\]
Fix $\gamma>0$ and set $\beta_c=1/(2\gamma)$.

\begin{itemize}
\item If the \emph{upper} envelope holds and $E(\eta)<\infty$, then all extinction conclusions of
Theorem~\ref{thm:general-gamma} remain valid with the same critical parameter $\beta_c$ and the same constants
$K^{\uparrow}_{\gamma,\beta}$ and $K^{\downarrow}_{\gamma,\beta}$ (or $K^{\downarrow}_{\gamma,\beta}(c_0)$ when $\gamma>1$).
In particular:
  \begin{itemize}
  \item for $\gamma>1$, Theorem~\ref{thm:general-gamma}\textup{(A)(i)} yields extinction;
  \item for $\gamma<1$, Theorem~\ref{thm:general-gamma}\textup{(B)(i)} yields extinction.
  \end{itemize}

\item If the \emph{lower} envelope holds and $\PP(\eta=0)<1$, then all survival conclusions of
Theorem~\ref{thm:general-gamma} remain valid with the same constants. In particular:
  \begin{itemize}
  \item for $\gamma>1$, Theorem~\ref{thm:general-gamma}\textup{(A)(ii)} gives survival with positive probability;
  \item for $\gamma<1$, Theorem~\ref{thm:general-gamma}\textup{(B)(ii)} gives survival with positive probability.
  \end{itemize}
\end{itemize}
\end{corollary}

\begin{remark}
Fix $\delta>0$ and define, for $x\ge 1$,
\[
L(x)\ :=\ \frac{\delta}{\bigl(1+\log x\bigr)^{1+\delta}}.
\]
Set, for $u\in[0,1)$,
\[
g(u)\ :=\ \frac{1}{1-u}\,L\!\Big(\frac{1}{1-u}\Big)
\ =\ \delta\,\frac{1}{1-u}\,\frac{1}{\bigl(1+\log \tfrac{1}{1-u}\bigr)^{1+\delta}},
\]
and define $g(1):=0$. Then $g\ge 0$ on $[0,1]$ and $\int_0^1 g(u)\,du=1$, so $g$ is a probability density on $[0,1]$. Moreover, $L$ is slowly varying at $+\infty$.

Perform the change of variables $x=1/(1-u)$, i.e.\ $u=1-1/x$ and $du=x^{-2}\,dx$. Then
\[
\int_0^1 g(u)\,du
\;=\;\int_0^1 \frac{1}{1-u}\,L\!\Big(\frac{1}{1-u}\Big)\,du
\;=\;\int_{1}^{\infty} \frac{L(x)}{x}\,dx
\;=\;\delta\int_{1}^{\infty} \frac{dx}{x\,(1+\log x)^{1+\delta}}.
\]
With $t=1+\log x$ (so $dt=dx/x$), the last integral becomes
\[
\delta\int_{1}^{\infty} t^{-(1+\delta)}\,dt
\;=\;\delta\Big[\,-\frac{1}{\delta}\,t^{-\delta}\,\Big]_{1}^{\infty}
\;=\;1.
\]
Hence $\int_0^1 f=1$. Nonnegativity is immediate since $\delta>0$ and $L\ge 0$.

For any fixed $t>0$,
\[
\frac{L(tx)}{L(x)}
\;=\;\Bigg(\frac{1+\log x}{1+\log(tx)}\Bigg)^{1+\delta}
\;=\;\Bigg(\frac{1+\log x}{1+\log x+\log t}\Bigg)^{1+\delta}
\;=\;\Bigg(1+\frac{\log t}{1+\log x}\Bigg)^{-(1+\delta)}
\xrightarrow[x\to\infty]{}1.
\]
Thus, $L$ is slowly varying at $+\infty$.

Assume that $\pi_{x,i}$ admits a density $g$ and that $\mathbb{P}(\eta=0)<1$. Then, by Corollary~\ref{cor:general-gamma-envelope}, survival holds for every $\gamma>0$.

\end{remark}

\begin{remark}
Assume that there exists $a<1$ such that $\PP(\pi\in[0,a])=1$ and $E(\eta)<\infty$. Then, by Corollary~\ref{cor:general-gamma-envelope}, extinction holds for every $\gamma>0$.
\end{remark}

The remainder of this paper is organized as follows. Section~\ref{sec:one_particle} states several results regarding the displacement of the random walk of a single particle on $\mathbb{Z}$; these estimates are important to help with the proof our main theorem. Section~\ref{results} is devoted to the proofs of our main results.
Section~\ref{sec:gamma>1} treats the case $\gamma>1$ and, combining
Propositions~\ref{prop:upper-gamma>1-EN-eps} and \ref{prop:lower-gamma>1-EN-eps}
with a dominated convergence argument, yields Theorem~\ref{thm:gamma>1-final-EN-eps}.
Section~\ref{sec:gamma<1} deals with the regime $0<\gamma<1$, proving
Theorem~\ref{thm:gamma<1-final-EN-eps} and, as a consequence, the general
extinction–survival dichotomy of Theorem~\ref{thm:general-gamma}.
Appendix~A: Auxiliary Results collects auxiliary estimates for the random walk
and for regularly varying functions that are used throughout these proofs.

\section{One–particle tail asymptotics}\label{sec:one_particle}

Under \eqref{eq:edge-behavior} we derive \emph{two–sided} asymptotics for the one–particle displacement tail. 


For a particle $i$ initially at $z\in\mathbb{Z}$, define the maximal one–sided and two–sided displacements during its lifetime:
\[
D^{\rightarrow}_{z,i}:=\max_{0\le n\le \Xi_{z,i}}\big(S^{z,i}_n-z\big),\,
D^{\leftarrow}_{z,i}:=\max_{0\le n\le \Xi_{z,i}}\big(z-S^{z,i}_n\big),\,
D^{*}_{z,i}:=\max_{0\le n\le \Xi_{z,i}}\big(|S^{z,i}_n-z|\big)=D^{\rightarrow}_{z,i}\vee D^{\leftarrow}_{z,i}.
\]

By homogeneity, write $D^{\rightarrow}$, $D^{\leftarrow}$, and $D^{*}$ for generic copies. By symmetry,
\begin{equation}\label{eq:dstar-half}
\mathbb{P}(D^{\rightarrow}\ge n)=\mathbb{P}(D^{\leftarrow}\ge n)\ \ge\ \tfrac12\,\mathbb{P}(D^{*}\ge n),\qquad n\in\mathbb{N}.
\end{equation}

Since, conditional on the survival parameters $\{\pi_{x,i}\}$, the lifetimes $\{\Xi_{x,i}\}$ are independent of the family of simple symmetric random walks $\{(S^{x,i}_n)_{n\ge 0}\}$, and since $\{\pi_{x,i}\}$ itself is independent of the family of simple symmetric random walks, we have that

\begin{equation}\label{eq:reduce-key-EN-eps}
\begin{aligned}
\PP(D^{\rightarrow}\ge n\mid \pi=p)&= \PP(\tau_n\le \Xi\mid \pi=p)= \sum_{k=0}^{\infty}\PP(\Xi\ge k,\ \tau_n=k\mid \pi=p)\\
&= \sum_{k=0}^{\infty}\PP(\Xi\ge k\mid \pi=p)\,\PP(\tau_n=k)= \sum_{k=0}^{\infty} p^{k^{\gamma}}\,\PP(\tau_n=k)
= \EE\!\left[p^{\tau_n^{\gamma}}\right], \qquad n\in\mathbb{N}.
\end{aligned}
\end{equation}

so the random walk enters only through the law of the hitting time $\tau_n$, while the lifetime acts as an “age filter”.

As noted in \cite{CarvalhoMachado2025} and  \cite{JGF}, survival or extinction of the frog model on $\mathbb{Z}$ depends crucially on the asymptotic behavior of $P(D^\rightarrow \ge n)$. Accordingly, below we establish several important results concerning the one-sided displacements of the particles, of which our main theorem is an almost direct application.

\begin{proposition}[Upper bound \(\gamma > 1\)]\label{prop:upper-gamma>1-EN-eps}
Assume \(\gamma>1\) and \eqref{eq:edge-behavior}. Then, for every $\varepsilon\in(0,(2\gamma\beta)\wedge\tfrac{1}{2})$ and \(C>1\) there exists $N_{\varepsilon,C}$ such that for all $n\ge N_{\varepsilon,C}$,
\begin{equation}\label{eq:upper-gamma>1-EN-eps}
n\,\PP(D^{\rightarrow}\ge n)\ \le\ n^{1-2\gamma\beta}L(n^{2\gamma})D^+_{\gamma,\beta,\varepsilon,C}+n2^{-n^\gamma},
\end{equation}

where 
\[
D^+_{\gamma,\beta,\epsilon,C}:=2\gamma (1+\varepsilon)^{2} c_{\varepsilon}^{-2\gamma\beta}C\int_{0}^{\infty}e^{-y}y^{2\gamma\beta-1}\max\{y^{\varepsilon},y^{-\varepsilon}\}dy
\]

and \(c_{\varepsilon}:=(1-\varepsilon)\Gamma\!\bigl(1-\tfrac{1}{2\gamma}\bigr)\sqrt{\tfrac{2}{\pi}}2^{-1/(2\gamma)}\).

\end{proposition}

\begin{proposition}[Lower bound \(\gamma>1\)]\label{prop:lower-gamma>1-EN-eps}
Assume \(\gamma>1\) and \eqref{eq:edge-behavior}. Then, for every \(\varepsilon\in(0,1)\), \(c_0>0\) and \(0<y_0<y_1<\infty\) there exists $N_{\varepsilon,c_0,y_0,y_1}$ such that for all $n\ge N_{\varepsilon,c_0}$,
\begin{equation}\label{eq:lower-gamma>1-EN-eps}
n\,\PP(D^{\rightarrow}\ge n)\ \ge n^{1-2\gamma\beta}L(n^{2\gamma})C^-_{\gamma,\beta,\varepsilon,c_0}(y_0,y_1),
\end{equation}
where
\[
C^-_{\gamma,\beta,\epsilon,c_0}(y_0,y_1)
:= (1-\varepsilon^2)\,2\gamma\,\theta\int_{y_0}^{y_1} y^{2\gamma\beta-1}\exp\!\big(-(1+\varepsilon)c_0^{\gamma}y^{2\gamma}\big)\,dy,
\qquad 0<y_0<y_1<\infty,
\]
and \(\theta :=\ \tfrac12\big(1-\Phi(1/\sqrt{c_0})\big)\in(0,1)\).
\end{proposition}

\begin{theorem}[Case \(\gamma>1\)]\label{thm:gamma>1-final-EN-eps}
Assume \(\gamma>1\) and \eqref{eq:edge-behavior}. Then, for every $c_0>0$, we have the following:
\begin{equation}\label{lim_gamma>1}
    \begin{split}
       &2\gamma\ \theta \int_{0}^{\infty}y^{2\gamma\beta-1}e^{-c_0^{\gamma}y^{2\gamma}}dy\le\,\liminf_{n\to \infty} \frac{n\,\PP(D^{\rightarrow}\ge n)}{n^{1-2\beta\gamma}L(n^{2\gamma})}\\
       &\le\,\limsup_{n\to \infty} \frac{n\,\PP(D^{\rightarrow}\ge n)}{n^{1-2\beta\gamma}L(n^{2\gamma})}\le\,2\gamma\ \Gamma(2\beta\gamma)2^{\beta(1-\gamma)}\Bigg(\frac{\Gamma(1-\frac{1}{2\gamma})}{\sqrt{\pi}}\Bigg)^{-2\gamma\beta}
    \end{split}
\end{equation}
where \(\theta:=\ \tfrac12\big(1-\Phi(1/\sqrt{c_0})\big)\in(0,1)\) and $\Phi$ is the standard Gaussian cdf. 
\end{theorem}

\begin{theorem}[Case \(\gamma<1\)]\label{thm:gamma<1-final-EN-eps}
Under \eqref{eq:edge-behavior} and \(\gamma<1\), we have the following:
\begin{equation}\label{lim_gamma_1}
    \begin{split}
       &2\gamma\ \Gamma(2\beta\gamma)2^{-\beta(1+\gamma)}\frac{\Gamma(\beta)}{\gamma\Gamma(\beta\gamma)}\le\,\liminf_{n\to \infty} \frac{n\,\PP(D^{\rightarrow}\ge n)}{n^{1-2\beta\gamma}L(n^{2\gamma})}\\
       &\le\,\limsup_{n\to \infty} \frac{n\,\PP(D^{\rightarrow}\ge n)}{n^{1-2\beta\gamma}L(n^{2\gamma})}\le\,2\gamma\ \Gamma(2\beta\gamma)2^{-\beta\gamma}\frac{\Gamma(\beta)}{\gamma\Gamma(\beta\gamma)}
    \end{split}
\end{equation}

\end{theorem}

\section{Proofs}\label{results}

We shall also use the following two properties in the proofs of our main results:

For a simple symmetric random walk on $\mathbb{Z}$, let \(\tau_n\ :=\ \inf\{k\ge0:\ S_k=n\}\) the first–passage time to $+n$. Then, \(\tau_n\) satisfies
\begin{equation}\label{sumindp}
\EE[s^{\tau_n}]=f(s)^n:=\Big(\tfrac{1-\sqrt{1-s^2}}{s}\Big)^{\!n}
\quad\text{and}\quad
\tau_n\ \stackrel{d}{=}\ \sum_{i=1}^{n} \tau_1^{i},\ \ \ \tau_1^{i}\ \text{i.i.d. with }\tau_1^{i}\stackrel{d}{=}\tau_1,
\end{equation}

see, e.g., \cite[Lemma 19, Theorem 33]{Konstantopoulos09}. 

In addition, by \eqref{eq:edge-behavior}, for every \(\varepsilon>0\) there exists \(h_{\varepsilon}>0\) such that, for all \(0<h\le h_{\varepsilon}\),
\begin{equation}\label{eq:edge-band}
(1-\varepsilon)\,h^{\beta-1}L(1/h)
\le f_\pi(1-h)\le
(1+\varepsilon)\,h^{\beta-1}L(1/h).
\end{equation}

We now proceed to the proofs.

\subsection{$\gamma>1$}\label{sec:gamma>1}

\begin{proof}[Proof of Proposition \ref{prop:upper-gamma>1-EN-eps}]
    Let \(\varepsilon\in \bigl(0,(2\gamma\beta)\wedge\tfrac{1}{2}\bigr)\). By \eqref{sumindp} and Lemmas~\ref{lem:tauber-Tgamma-EN-eps}–\ref{lem:convex-ineq-EN-eps},
there exists $a_\varepsilon>0$ such that, for all $0<a\le a_\varepsilon$,
\[
\EE(e^{-a\,\tau_n^\gamma})\ \le\ \big(\EE(e^{-a \tau_1^\gamma})\big)^n\ \le\ \big(1-c_1(\varepsilon) a^{1/(2\gamma)}\big)^n
\ \le\ \exp\big(-c_1(\varepsilon)\,n\,a^{1/(2\gamma)}\big).
\]
where \(c_1(\varepsilon)=(1-\varepsilon)\Gamma\!\bigl(1-\tfrac{1}{2\gamma}\bigr)\sqrt{\tfrac{2}{\pi}}\). 

Let \(a(1-h)=-\log p\); for \(h\in \bigl[0,\tfrac{1}{2}\bigr]\) one has \(a(1-h)\in [h,2h]\), hence on \(h\in\bigl(0,h_\varepsilon\wedge\tfrac{1}{2}\bigr]\),
\begin{equation}\label{gac1}
\EE\!\big(e^{-a(1-h)\,\tau_n^\gamma}\big)\le \exp\{-c\,n\, h^{1/(2\gamma)}\} 
\end{equation}
with \(c_{\varepsilon}=c_1(\varepsilon)\,2^{-1/(2\gamma)}\). Moreover, by (\ref{eq:edge-band}) there exists \(h_{\varepsilon}\) such that 
\begin{equation}\label{gac2}
f_\pi(1-h)\ \le\
(1+\varepsilon)\,h^{\beta-1}L(1/h),\qquad 0<h\le h_\varepsilon .
\end{equation}

Let \(h_{2,\varepsilon}:=h_{\varepsilon}\wedge\tfrac{a_{\varepsilon}}{2}\wedge\tfrac{1}{2}\). Then
\begin{equation}
    \begin{split}
        \int_{0}^{1}\EE\!\big(p^{\tau_n^{\gamma}}\big)f_{\pi}(p)\,dp
        &=\int_{0}^{1}\EE\!\big(e^{-a(p)\tau_n^{\gamma}}\big)f_{\pi}(p)\,dp\\
        &=\int_{0}^{1-h_{2,\varepsilon}}\EE\!\big(e^{-a(p)\tau_n^{\gamma}}\big)f_{\pi}(p)\,dp+\int_{1-h_{2,\varepsilon}}^{1}\EE\!\big(e^{-a(p)\tau_n^{\gamma}}\big)f_{\pi}(p)\,dp\\
        &=\int_{h_{2,\varepsilon}}^{1}\EE\!\big(e^{-a(1-h)\tau_n^{\gamma}}\big)f_{\pi}(1-h)\,dh+\int_{0}^{h_{2,\varepsilon}}\EE\!\big(e^{-a(1-h)\tau_n^{\gamma}}\big)f_{\pi}(1-h)\,dh\\
        &:=I_{1,n}+I_{2,n}\,.
    \end{split}
\end{equation}

Furthermore, since \(\tau_{n}\ge n\) a.s., one has
\begin{equation}\label{limI_1}
        I_{1,n}\le \int_{h_{2,\varepsilon}}^{1}\EE\!\big(e^{-a(1-h_{2,\varepsilon})\tau_n^{\gamma}}\big)f_{\pi}(1-h)\,dh
        \le e^{-a(1-h_{2,\varepsilon})n^{\gamma}}\int_{h_{2,\varepsilon}}^{1}f_{\pi}(1-h)\,dh
        \xrightarrow[n\to\infty]{}0.
\end{equation}

We now estimate the order of \(I_{2,n}\). By \eqref{gac1}, \eqref{gac2} and making the change of variable \(y=c_{\varepsilon}\,n\,h^{1/(2\gamma)}\),
\begin{equation}
    \begin{split}
        I_{2,n} &\le (1+\varepsilon) \int_{0}^{h_{2,\varepsilon}} e^{-c_{\varepsilon}\,n \,h^{1/(2\gamma)}}h^{\beta-1}L\!\Big(\tfrac{1}{h}\Big)\,dh\\
        &\le 2\gamma\,(1+\varepsilon) \,(n c_{\varepsilon})^{-2\gamma\beta}\int_{0}^{c_{\varepsilon}n h_{2,\varepsilon}}e^{-y}\,y^{2\gamma\beta-1}\,
        L\!\Big(\big(\tfrac{c_{\varepsilon}n}{y}\big)^{2\gamma}\Big)\,dy\,.
    \end{split}
\end{equation}

By Lemma~\ref{lem:potter-EN-eps}, equation (\ref{eq:potter-EN-eps}), given \(C>1\) there exists \(X(\varepsilon,C)\ge 1\) such that, for all \(x\ge X(\varepsilon,C)\) and \(z>0\),
\[
\frac{L(xz)}{L(x)}\ \le\ C\,\max\{z^\varepsilon,z^{-\varepsilon}\}.
\]
Setting \(x=(c_{\varepsilon}n)^{2\gamma}\) and \(z=\tfrac{1}{y}\) we get
\[
\int_{0}^{c_{\varepsilon}n h_{2,\varepsilon}}e^{-y}\,y^{2\gamma\beta-1}\,
L\!\Big(\big(\tfrac{c_{\varepsilon}n}{y}\big)^{2\gamma}\Big)\,dy
\ \le\ 
C\int_{0}^{c_{\varepsilon}n h_{2,\varepsilon}}e^{-y}\,y^{2\gamma\beta-1}\,
L\big((c_{\varepsilon}n)^{2\gamma}\big)\,\max\{y^{\varepsilon},y^{-\varepsilon}\}\,dy,
\]
and by Lemma~\ref{lem:potter-EN-eps}, equation (\ref{eq:UCT-EN-eps}), there exists \(N(\varepsilon,C)\) such that, if \(n\ge N(\varepsilon,C)\),
\begin{equation}\label{limI_2}
\begin{split}
    I_{2,n}&\le 2\gamma\,(1+\varepsilon)^{2} \,(n c_{\varepsilon})^{-2\gamma\beta}L(n^{2\gamma})\,C
    \int_{0}^{c_{\varepsilon}n h_{2,\varepsilon}}e^{-y}\,y^{2\gamma\beta-1}\max\{y^{\varepsilon},y^{-\varepsilon}\}\,dy\\
    &\le 2\gamma\,(1+\varepsilon)^{2} \,(n c_{\varepsilon})^{-2\gamma\beta}L(n^{2\gamma})\,C
    \int_{0}^{\infty}e^{-y}\,y^{2\gamma\beta-1}\max\{y^{\varepsilon},y^{-\varepsilon}\}\,dy\,.
\end{split}
\end{equation}

Therefore, combining \eqref{limI_1} (\(h_{2,\varepsilon} \le \frac{1}{2}\)) and \eqref{limI_2} yields \eqref{eq:upper-gamma>1-EN-eps}.

\end{proof}

\begin{proof}[Proof of Proposition \ref{prop:lower-gamma>1-EN-eps}]

Let \(0<y_0<y_1<\infty\). Let \(\varepsilon>0\). By (\ref{eq:edge-band}) there exists \(h_\varepsilon>0\) such that:
\[
(1-\varepsilon)\,h^{\beta-1}L(1/h)
\ \le\ f_\pi(1-h),\,\phantom{q}h\leq h_{\varepsilon}
\]
Moreover, since \(a(h):=-\log(h)\) satisfies \(\frac{a(1-h)}{h}\overset{h\to 0}{\to}1\), there exists \(h_{2,\varepsilon}>0\) such that 
\[
a(1-h)\le (1+\varepsilon)h,\,\phantom{a}h\le h_{2,\varepsilon}
\]

Evidently, there exists \(\,N_{\varepsilon,y_0,y_1}\in \mathbb{N}\) such that if \(n\geq N_{\varepsilon,y_0,y_1}\), then \([(\frac{y_0}{n})^{2\gamma},(\frac{y_1}{n})^{2\gamma}]\subset [0,h_{\varepsilon}\wedge h_{2,\varepsilon}]\).

Restrict \(p\) to \([1-(y_1/n)^{2\gamma},\,1-(y_0/n)^{2\gamma}]\) (i.e.,\ \(h\in[(y_0/n)^{2\gamma},(y_1/n)^{2\gamma}]\)). 

Let \(c_0>0\). By Lemma~\ref{lem:inv-principle-EN-eps} there exists \(N_{c_0}\in \mathbb{N}\) such that if \(n\ge N_{c_0}\) then \(P(\tau_n< c_0 n^2)\ge\theta\), where \(\theta:=\ \tfrac12\big(1-\Phi(1/\sqrt{c_0})\big)\in(0,1)\). Let \(n\ge N_{\varepsilon,y_0,y_1}\vee N_{c_0}\). Then, it is evident that, on \(\{\tau_n< c_0 n^2\}\),
\[
\EE(e^{-a\,\tau_n^\gamma})\ge \theta\,e^{-a\,c_0^\gamma n^{2\gamma}}.
\]
Then, performing the change of variables \(h=(y/n)^{2\gamma}\), we obtain:
\begin{equation}
    \begin{split}
        \PP(D^{\rightarrow}\ge n)&\geq \int_{1-(\frac{y_0}{n})^{2\gamma}}^{1-(\frac{y_1}{n})^{2\gamma}}E(e^{-a(p)\tau_n^{\gamma}})f_{\pi}(p)dp=\theta\int_{(\frac{y_0}{n})^{2\gamma}}^{(\frac{y_1}{n})^{2\gamma}}e^{-a(1-h)c_0^{\gamma}n^{2\gamma}}f_{\pi}(1-h)dh\\
        &\geq(1+\varepsilon)\ \theta\int_{(\frac{y_0}{n})^{2\gamma}}^{(\frac{y_1}{n})^{2\gamma}}e^{-(1+\varepsilon)hc_0^{\gamma}n^{2\gamma}}h^{\beta-1}L(1/h)dh\\
        &=(1+\varepsilon)\ 2\gamma n^{-2\gamma\beta}\theta\int_{y_0}^{y_1}e^{-(1+\varepsilon)c_0^{\gamma}y^{2\gamma}}y^{2\gamma\beta-1}L(\frac{n^{2\gamma}}{y^{2\gamma}})dy
    \end{split}
\end{equation}

Lemma~\ref{lem:potter-EN-eps}, equation (\ref{eq:UCT-EN-eps}), given \(\varepsilon>0\) there exists \(N_{2,\varepsilon,y_0,y_1}\in \mathbb{N}\) such that if \(n\geq N_{2,\varepsilon,y_0,y_1}\) then \(\sup_{[y_0,y_1]}L(\frac{n^{2\gamma}}{y^{2\gamma}})\geq (1-\varepsilon)L(n^{2\gamma})\). Therefore, if \(n\geq N_{\varepsilon,c_0,y_0,y_1}:=N_{\varepsilon,y_0,y_1}\vee N_{c_0}\vee N_{2,\varepsilon,y_0,y_1}\), then (\ref{eq:lower-gamma>1-EN-eps}) follows.

\end{proof}

\begin{proof}[Proof of Theorem \ref{thm:gamma>1-final-EN-eps}] By Propositions \ref{prop:upper-gamma>1-EN-eps} and \ref{prop:lower-gamma>1-EN-eps}, together with the Dominated Convergence Theorem---letting \(\varepsilon \downarrow 0\) and subsequently \(C \downarrow 1\), \(y_0 \downarrow 0\), and \(y_1 \uparrow \infty\)---we obtain \eqref{lim_gamma>1}.

\end{proof}

\subsection{$\gamma<1$}\label{sec:gamma<1}

\begin{proof}[Proof of Theorem \ref{thm:gamma<1-final-EN-eps}]
Let \(C_{+}>0\) and \(a(p):=-\log p\). By Lemma~\ref{lem:stable-rep-EN-eps} and Tonelli’s theorem 
\[
\begin{aligned}
n\,\PP(D^{\rightarrow}\ge n)
&=\int_0^1 n\,\EE\!\big[e^{-a(p)\,\tau_n^\gamma}\big]\, f_\pi(p)\,dp
=\int_0^1 n\,\Big(\int_0^\infty \EE\!\big[e^{-u\,\tau_n}\big]\ \nu_{\gamma,a(p)}(du)\Big)\, f_\pi(p)\,dp\\
&\stackrel{\eqref{eq:scale-EN-eps}}{=}\int_0^1 n\,\Big(\int_0^\infty \EE\!\big[e^{-\,(a(p)^{1/\gamma}v)\,\tau_n}\big]\ \nu_{\gamma,1}(dv)\Big)\, f_\pi(p)\,dp\\
&=\int_0^\infty \Big(\int_0^1 n\,\EE\!\big[e^{-\,(a(p)^{1/\gamma}v)\,\tau_n}\big]\, f_\pi(p)\,dp\Big)\,\nu_{\gamma,1}(dv),\\
&=\int_0^\infty \underbrace{\Big(\int_0^1 n\,f(e^{-\,a(p)^{1/\gamma}v})^{n}\, f_\pi(p)\,dp\Big)}_{=:T_n(v)}\,\nu_{\gamma,1}(dv),
\end{aligned}
\]

Fix \(\varepsilon\in(0,\beta\wedge1)\) and  \(\delta\in (0,1)\). Split
\[
T_n(v)=\underbrace{\int_{1-\delta}^1 n\, f\!\big(e^{-a(p)^{1/\gamma}v}\big)^{n} f_\pi(p)\,dp}_{=:I_{n,\delta,1}(v)}
+\underbrace{\int_0^{1-\delta} n\, f\!\big(e^{-a(p)^{1/\gamma}v}\big)^{n} f_\pi(p)\,dp}_{=:I_{n,\delta,2}(v)}.
\]
For \(I_{n,\delta,2}(v)\), since \(a(1-\delta)>0\). Then, for every \(v>0\),
\[
I_{n,\delta,2}(v)\ \le\ n\, f\!\big(e^{-a(1-\delta)^{1/\gamma}v}\big)^{\!n} \int_0^{1-\delta} f_\pi(p)\,dp
\ \le\ n\,\, f\!\big(e^{-a(1-\delta)^{1/\gamma}v}\big)^{n}.
\]
Integrating this bound in \(v\) against \(\nu_{\gamma,1}(dv)\) and invoking Lemma~\ref{lem:stable-rep-EN-eps} yields
\begin{equation}\label{eq:I2-small-integrated}
\int_0^\infty I_{n,\delta,2}(v)\,\nu_{\gamma,1}(dv)
\ \le\ n\, \int_0^\infty f\!\big(e^{-a(1-\delta)^{1/\gamma}v}\big)^{n}\,\nu_{\gamma,1}(dv)
\ =\ n\,\,\EE\!\left[e^{-a(1-\delta)\,\tau_n^{\gamma}}\right].
\end{equation}
Since \(\tau_n\ge n\) a.s., we have \(\EE[e^{-a(1-\delta) \tau_n^\gamma}]\le e^{-a(1-\delta) n^\gamma}\). Therefore,

\begin{equation}\label{cotaIn,delta,2}
    \int_0^\infty I_{n,\delta,2}(v)\,\nu_{\gamma,1}(dv)\le n e^{-a(1-\delta) n^\gamma}
\end{equation}

Write \(p=1-h\). For \(h\in(0,\tfrac12)\), we have the elementary bounds
\begin{equation}\label{eq:log-ineq}
h\ \le\ a(1-h)=-\log(1-h)\ \le\ \frac{h}{1-h}\ \le\ 2h
\end{equation}

Now, by Lemma~\ref{lem:two-sided-exp-EN}, there exist \(u_\varepsilon\in(0,1)\) such that, whenever
\[
a(1-h)^{1/\gamma}v\ \le\ u_\varepsilon,
\]
we have \(e^{-(1+\varepsilon)\sqrt{2u}}\le f(e^{-u})\le e^{-(1-\varepsilon)\sqrt{2u}}\). 

Now, since \(\frac{a(1-h)}{h}\overset{h\to 0}{\to}1\), there exists \(\delta_\varepsilon>0\) such that if \(h\in (0,\delta_\varepsilon)\), then, 
\[
a(1-h)^{1/\gamma}v\leq u_{\varepsilon},\,\forall\, 0<v\le C_{+}
\]

Then, by \eqref{eq:log-ineq},
\[
e^{-\,(1+\varepsilon)\sqrt{2}\,n\,v^{1/2} (2h)^{1/(2\gamma)}}\ \le\ f\!\big(e^{-a(1-h)^{1/\gamma}v}\big)^{n}
\ \le\ e^{-\,(1-\varepsilon)\sqrt{2}\,n\,v^{1/2} h^{1/(2\gamma)}},\, h\in(0,\delta_\varepsilon),\,v\le C_{+}.
\]
Moreover, there exists \(h_{\varepsilon}\) such that:

\[
(1-\varepsilon)\,h^{\beta-1}L(1/h)
\ \le\ f_\pi(1-h)\ \le\
(1+\varepsilon)\,h^{\beta-1}L(1/h),\qquad 0<h\le h_\varepsilon.
\]

Let \(h\in (0,\eta_{\varepsilon})\) with \(\eta_{\varepsilon}:=\delta_{\varepsilon}\wedge h_{\varepsilon}\). 

Suppose that \(0<C_{-}\leq v\leq C_{+}\) for some positive constant \(C_{-}\). Define \(\kappa(v):=2^{1/2((\gamma+1)/\gamma)}\,v^{1/2}\). Since \(C_{-}\leq v\leq C_{+}\), there exists \(N_{\varepsilon,C_{-},C_{+}}\in \mathbb{N}\) such that if \(n\ge N_{\varepsilon,C_{-},C_{+}}\) then \((1+\varepsilon)\,\kappa(v)\,n\,\eta_{\varepsilon}^{1/(2\gamma)}>C_{+}\).
Then, performing the change of variables
\begin{equation*}
y:=(1+\varepsilon)\,\kappa(v)\,n\,h^{1/(2\gamma)}
\quad\Longleftrightarrow\quad
h=\Big(\frac{y}{(1+\varepsilon)\,\kappa(v)\,n}\Big)^{2\gamma},\qquad
dh=(2\gamma)\,((1+\varepsilon)\,\kappa(v)\,n)^{-2\gamma}\,y^{2\gamma-1}dy,
\end{equation*}
we obtain
\begin{equation*}
    \begin{split}
        &I_{n,\eta_{\varepsilon},1}(v)=\int_{0}^{\eta_{\varepsilon}} n\, f\!\big(e^{-a(1-h)^{1/\gamma}v}\big)^{n} f_\pi(1-h)\,dh\\
        &\geq (1-\varepsilon)\int_{0}^{\eta_{\varepsilon}} n\, e^{-\,(1+\varepsilon)\sqrt{2}\,n\,v^{1/2} (2h)^{1/(2\gamma)}}h^{\beta-1}L(1/h)dh\\
        &=(1-\varepsilon)\ (2\gamma)n\,((1+\varepsilon)\,\kappa(v)\,n)^{-2\beta\gamma}\int_{0}^{(1+\varepsilon)\,\kappa(v)\,n\,\eta_{\varepsilon}^{1/(2\gamma)}} e^{-y}y^{2\beta\gamma-1}L\Bigg(\Bigg(\frac{(1+\varepsilon)\kappa(v)}{y}\Bigg)^{2\gamma}n^{2\gamma}\Bigg)\, \,dy\\
        &\geq (1-\varepsilon)\ (2\gamma)n\,((1+\varepsilon)\,\kappa(v)\,n)^{-2\beta\gamma}\int_{C_{-}}^{C_{+}} e^{-y}y^{2\beta\gamma-1}L\Bigg(\Bigg(\frac{(1+\varepsilon)\kappa(v)}{y}\Bigg)^{2\gamma}n^{2\gamma}\Bigg)\, \,dy
    \end{split}
\end{equation*}

Since there exist constants \(0<a<b<\infty\) such that \(\Big(\frac{(1+\varepsilon)\kappa(v)}{y}\Big)^{2\gamma}\in [a,b]\) for all \(y,v\in[C_{-},C_{+}]\), then by Lemma~\ref{lem:potter-EN-eps}, equation (\ref{eq:UCT-EN-eps}), there exists \(N_{2,\varepsilon,C_{-},C_{+}}\in\mathbb{N}\) such that if \(n\ge N_{2,\varepsilon,C_{-},C_{+}}\vee N_{\varepsilon,C_{-},C_{+}}\) we have
\begin{equation}
    I_{n,\eta_{\varepsilon},1}(v)\geq (1-\varepsilon)^2(2\gamma)n^{1-2\beta\gamma}L(n^{2\gamma})\,((1+\varepsilon)\,\kappa(v))^{-2\beta\gamma}\int_{C_{-}}^{C_{+}} e^{-y}y^{2\beta\gamma-1} \,dy
\end{equation}

Therefore, 
\begin{equation}\label{lim_inf_Dn}
\begin{split}
 &n\,\PP(D^{\rightarrow}\ge n)\ge  \int_{C_{-}}^{C_{+}} I_{n,\eta_{\varepsilon},1}(v)\,\nu_{\gamma,1}(dv) \\ 
 &\geq (1-\varepsilon)^2(2\gamma)n^{1-2\beta\gamma}L(n^{2\gamma})\,(1+\varepsilon)^{-2\beta\gamma}\int_{C_{-}}^{C_{+}} e^{-y}y^{2\beta\gamma-1} \,dy\int_{C_{-}}^{C_{+}}\kappa(v)^{-2\beta\gamma}\,\nu_{\gamma,1}(dv)
\end{split}    
\end{equation}

Now, for \(0<v\leq C_{+}\). Define \(\rho(v):=\sqrt{2}v^{1/2}\) and perform the change of variables
\begin{equation*}
y:=(1-\varepsilon)\,\rho(v)\,n\,h^{1/(2\gamma)}\quad\Longleftrightarrow\quad
h=\Big(\frac{y}{(1-\varepsilon)\,\rho(v)\,n}\Big)^{2\gamma},\quad
dh=(2\gamma)\,((1-\varepsilon)\,\rho(v)\,n)^{-2\gamma}\,y^{2\gamma-1}\,dy,
\end{equation*}

We have
\begin{equation*}
    \begin{split}
        &I_{n,\eta_{\varepsilon},1}(v)=\int_{0}^{\eta_{\varepsilon}} n\, f\!\big(e^{-a(1-h)^{1/\gamma}v}\big)^{n} f_\pi(1-h)\,dh\\
        &\leq (1+\varepsilon)\int_{0}^{\eta_{\varepsilon}} n\, e^{-\,(1-\varepsilon)\sqrt{2}\,n\,v^{1/2} h^{1/(2\gamma)}}h^{\beta-1}L(1/h)dh\\
        &=(1+\varepsilon)\ (2\gamma)n\,((1-\varepsilon)\,\rho(v)\,n)^{-2\beta\gamma}\int_{0}^{(1-\varepsilon)\,\rho(v)\,n\,\eta_{\varepsilon}^{1/(2\gamma)}} e^{-y}y^{2\beta\gamma-1}L\Bigg(\Bigg(\frac{(1-\varepsilon)\rho(v)}{y}\Bigg)^{2\gamma}n^{2\gamma}\Bigg)\, \,dy\\
        &\leq (1+\varepsilon)\ (2\gamma)n^{1-2\beta\gamma}\,((1-\varepsilon)\,\rho(v))^{-2\beta\gamma}\int_{0}^{\infty} e^{-y}y^{2\beta\gamma-1}L\Bigg(\Bigg(\frac{(1-\varepsilon)\rho(v)}{y}\Bigg)^{2\gamma}n^{2\gamma}\Bigg)\, \,dy
    \end{split}
\end{equation*}

By Lemma~\ref{lem:potter-EN-eps}, equation (\ref{eq:potter-EN-eps}), given \(C>1\) there exists \(X(\varepsilon,C)\ge 1\) such that, for all \(x\ge X(\varepsilon,C)\) and \(z>0\),
\[
\frac{L(xz)}{L(x)}\ \le\ C\,\max\{z^\varepsilon,z^{-\varepsilon}\}.
\]

Therefore, for \(n\) sufficiently large,
\begin{equation}\label{I_1_b}
    \begin{split}
       &I_{n,\eta_{\varepsilon},1}(v)\le  (1+\varepsilon)C(2\gamma)n^{1-2\beta\gamma}L(n^{2\gamma})\,((1-\varepsilon)\,\rho(v))^{-2\beta\gamma}\\
       &\cdot \int_{0}^{\infty} e^{-y}y^{2\beta\gamma-1}\,\max\Bigg\{\Bigg(\frac{(1-\varepsilon)\rho(v)}{y}\Bigg)^{2\gamma\varepsilon},\Bigg(\frac{(1-\varepsilon)\rho(v)}{y}\Bigg)^{-2\gamma\varepsilon}\Bigg\} \,dy\\
       &\le  (1+\varepsilon)C(2\gamma)n^{1-2\beta\gamma}L(n^{2\gamma})\,(1-\varepsilon)^{-2\beta\gamma-2\gamma\varepsilon}\rho(v)^{-2\beta\gamma}\max\{\rho(v)^{2\gamma\varepsilon},\rho(v)^{-2\gamma\varepsilon}\}\\
       &\cdot \int_{0}^{\infty} e^{-y}y^{2\beta\gamma-1}\,\max\{y^{2\gamma\varepsilon},y^{-2\gamma\varepsilon}\} \,dy\\
    \end{split}
\end{equation}

Additionally, it is straightforward to verify that if \(v>C_{+}\) then \(I_{n,\eta_{\varepsilon},1}(v)\le I_{n,\eta_{\varepsilon},1}(C_{+})\). Therefore, by (\ref{eq:I2-small-integrated}) and (\ref{I_1_b})
\begin{equation}\label{lim_sup_Dn}
    \begin{split}
    &n\,\PP(D^{\rightarrow}\ge n)=\int_0^\infty \Big(\int_0^1 n\,f(e^{-\,a(p)^{1/\gamma}v})^{n}\, f_\pi(p)\,dp\Big)\,\nu_{\gamma,1}(dv)\\
    &=\int_0^\infty \Big(I_{n,\eta_{\varepsilon},1}(v)+I_{n,\eta_{\varepsilon},2}(v)\Big)\,\nu_{\gamma,1}(dv)\\
    &\leq \int_0^{C_{+}} I_{n,\eta_{\varepsilon},1}(v)\,\nu_{\gamma,1}(dv)+I_{n,\eta_{\varepsilon},1}(C_{+})\nu_{\gamma,1}((C_{+},\infty))+ne^{-a(1-\eta_{\epsilon})n^{\gamma}}\\
    &\le  (1+\varepsilon)C(2\gamma)n^{1-2\beta\gamma}L(n^{2\gamma})\,(1-\varepsilon)^{-2\beta\gamma-2\gamma\varepsilon} \int_{0}^{\infty} e^{-y}y^{2\beta\gamma-1}\,\max\{y^{2\gamma\varepsilon},y^{-2\gamma\varepsilon}\} \,dy\\
       &\cdot\Bigg(\int_0^{C_{+}}\rho(v)^{-2\beta\gamma}\max\{\rho(v)^{2\gamma\varepsilon},\rho(v)^{-2\gamma\varepsilon}\}\,\nu_{\gamma,1}(dv)\\
       &+\rho(C_{+})^{-2\beta\gamma}\max\{\rho(C_{+})^{2\gamma\varepsilon},\rho(C_{+})^{-2\gamma\varepsilon}\}\nu_{\gamma,1}((C_{+},\infty))\Bigg)+ne^{-a(1-\eta_{\epsilon})n^{\gamma}}\\
    \end{split}
\end{equation}

Therefore, by (\ref{lim_inf_Dn}) and (\ref{lim_sup_Dn})
\small
\begin{equation*}
    \begin{split}
       &(1-\varepsilon)^2(2\gamma)\,(1+\varepsilon)^{-2\beta\gamma}\int_{C_{-}}^{C_{+}} e^{-y}y^{2\beta\gamma-1} \,dy\int_{C_{-}}^{C_{+}}\kappa(v)^{-2\beta\gamma}\,\nu_{\gamma,1}(dv)\le\,\liminf_{n\to \infty} \frac{n\,\PP(D^{\rightarrow}\ge n)}{n^{1-2\beta\gamma}L(n^{2\gamma})}\\
       &\le\,\limsup_{n\to \infty} \frac{n\,\PP(D^{\rightarrow}\ge n)}{n^{1-2\beta\gamma}L(n^{2\gamma})}\le\,(1+\varepsilon)C(2\gamma)\,(1-\varepsilon)^{-2\beta\gamma-2\gamma\varepsilon} \int_{0}^{\infty} e^{-y}y^{2\beta\gamma-1}\,\max\{y^{2\gamma\varepsilon},y^{-2\gamma\varepsilon}\} \,dy\\
       &\,\cdot\Bigg(\int_0^{C_{+}}\rho(v)^{-2\beta\gamma}\max\{\rho(v)^{2\gamma\varepsilon},\rho(v)^{-2\gamma\varepsilon}\}\,\nu_{\gamma,1}(dv)\\
        &\hspace{5mm}+\rho(C_{+})^{-2\beta\gamma}\max\{\rho(C_{+})^{2\gamma\varepsilon},\rho(C_{+})^{-2\gamma\varepsilon}\}\nu_{\gamma,1}((C_{+},\infty))\Bigg)
    \end{split}
\end{equation*}
\normalsize

Finally, by the Dominated Convergence Theorem—letting \(\varepsilon \downarrow 0\) and subsequently \(C \downarrow 1\), \(C_{-} \downarrow 0\), and \(C_{+} \uparrow \infty\) and \eqref{eq:neg-moment-EN-eps}, we obtain:

\begin{equation*}
    \begin{split}
       &2\gamma\ \Gamma(2\beta\gamma)2^{-\beta(1+\gamma)}\frac{\Gamma(\beta)}{\gamma\Gamma(\beta\gamma)}\le\,\liminf_{n\to \infty} \frac{n\,\PP(D^{\rightarrow}\ge n)}{n^{1-2\beta\gamma}L(n^{2\gamma})}\\
       &\le\,\limsup_{n\to \infty} \frac{n\,\PP(D^{\rightarrow}\ge n)}{n^{1-2\beta\gamma}L(n^{2\gamma})}\le\, 2\gamma\ \Gamma(2\beta\gamma)2^{-\beta\gamma}\frac{\Gamma(\beta)}{\gamma\Gamma(\beta\gamma)}
    \end{split}
\end{equation*}

Hence, we conclude (\ref{lim_gamma_1}).

\end{proof}

\begin{proof}[Proof of Theorem \ref{thm:general-gamma}]
Applying \cite[Proposition~1.2]{CarvalhoMachado2025} to the bounds in \eqref{lim_gamma>1}–\eqref{lim_gamma_1} establishes the claim.
\end{proof}

\section{Appendix A: Auxiliary Results}\label{sec:auxiliary-results}

We begin with several lemmas that provide estimates for the asymptotic behavior of
\(n\,\PP(D^{\rightarrow}\ge n)\) as \(n\to\infty\).

\begin{lemma}\label{lem:two-sided-exp-EN}
Let \(f:(0,1]\to(0,1]\) be defined by
\[
f(x)\;:=\;\frac{1-\sqrt{1-x^{2}}}{x}.
\]
Then there exist constants \(C>0\) and \(U\in(0,1)\) such that, for all \(u\in(0,U]\),
\begin{equation}\label{eq:logf-bound-raw}
\bigl|\log f(e^{-u})+\sqrt{2u}\bigr|\ \le\ C\,u^{3/2}.
\end{equation}
Consequently, for every \(\varepsilon\in(0,1)\) there exists \(u_\varepsilon\in(0,U]\) such that, for all \(u\in(0,u_\varepsilon]\),
\begin{equation}\label{eq:two-sided-eps-EN-eps}
e^{-(1+\varepsilon)\sqrt{2u}}
\;\le\;
f(e^{-u})
\;\le\;
e^{-(1-\varepsilon)\sqrt{2u}}.
\end{equation}
Moreover, one may choose explicitly
\begin{equation}\label{eq:explicit-ueps}
u_\varepsilon\ :=\ \min\!\left\{\,U,\ \frac{2\,\varepsilon^2}{C^2}\right\}.
\end{equation}
\end{lemma}

\begin{proof}
Expand \(e^{-2u}\) at \(u=0\) and set
\[
y(u):=1-e^{-2u}=2u+O(u^{2})=2u\,(1+O(u)).
\]
By Lemma~2.1(i) in \cite{JGF} with \(\gamma=\tfrac12\), for \(u\) near \(0\) we have
\[
\sqrt{1+u}=1+\tfrac12 u+O(u^{2}).
\]
In particular,
\[
\sqrt{y(u)}=\sqrt{2u}\,\bigl(1+O(u)\bigr).
\]
Moreover, by Lemma~2.1(iii) in \cite{JGF}, for \(u\) near \(0\),
\[
\log(1-u)=-u-\tfrac12 u^{2}+O(u^{3}).
\]
As \(u\downarrow0\),
\begin{equation*}
\begin{split}
  \log(f(e^{-u}))&=\log\Bigg(\frac{1-\sqrt{1-e^{-2u}}}{e^{-u}}\Bigg)= \log(1-\sqrt{y(u)})+u=-\sqrt{y(u)}-\frac{1}{2}y(u)+O(y(u)^{\frac{3}{2}})+u\\
  &=-\sqrt{2u}(1+\frac{1}{2}O(u)+O(u^2)) -u(1+O(u))+O(y(u)^{\frac{3}{2}})+u
\end{split}
\end{equation*}
since \(y(u)=O(u)\) and \(O(y(u)^{3/2})=O(u^{3/2})\). This proves \eqref{eq:logf-bound-raw}. Exponentiating \eqref{eq:logf-bound-raw} yields
\[
e^{-\sqrt{2u}}\,e^{-C u^{3/2}}
\ \le\ f(e^{-u})\ \le\
e^{-\sqrt{2u}}\,e^{C u^{3/2}}.
\]
Fix \(\varepsilon\in(0,1)\) and choose \(u_\varepsilon\) as in \eqref{eq:explicit-ueps}. If \(0<u\le u_\varepsilon\le 1\), then \(C u^{3/2}\le \varepsilon \sqrt{2u}\), and substituting into the last estimate gives \eqref{eq:two-sided-eps-EN-eps}.
\end{proof}

\begin{lemma}\label{lem:stable-rep-EN-eps}
Fix $0<\gamma< 1$ and $a>0$. Then there exists a \emph{probability} measure $\nu_{\gamma,a}$ on $(0,\infty)$ such that, for all $t\ge 0$,
\begin{equation}\label{eq:bernstein-EN-eps}
e^{-a\,t^{\gamma}}=\int_{0}^{\infty} e^{-u t}\,\nu_{\gamma,a}(du).
\end{equation}
Moreover, the family $(\nu_{\gamma,a})_{a>0}$ satisfies the exact scaling
\begin{equation}\label{eq:scale-EN-eps}
\nu_{\gamma,a}(du)=a^{-1/\gamma}\,\nu_{\gamma,1}\!\left(d\big(u/a^{1/\gamma}\big)\right),
\end{equation}
and, for every $\theta>0$,
\begin{equation}\label{eq:neg-moment-EN-eps}
\int_{0}^{\infty} u^{-\theta}\,\nu_{\gamma,1}(du)
=\mathbb{E}\!\left[S_1^{-\theta}\right]
=\frac{1}{\gamma}\,\frac{\Gamma(\theta/\gamma)}{\Gamma(\theta)}<\infty,
\end{equation}
where $S=(S_t)_{t\ge0}$ denotes a (one–sided) $\gamma$–stable subordinator with Laplace transform $\mathbb{E}[e^{-\lambda S_t}]=e^{-t\lambda^\gamma}$.
\end{lemma}

\begin{proof}

Let $S=(S_t)_{t\ge 0}$ be a \emph{$\gamma$-stable subordinator} (nondecreasing Lévy process) in the standard normalization
\[
\mathbb{E}\!\left[e^{-\lambda S_t}\right] \;=\; \exp\!\big(-t\,\lambda^{\gamma}\big), \qquad \lambda,t>0,\ \gamma\in(0,1).
\]
(For comparison, \cite{LalleySubordinators} writes $\mathbb{E}[e^{-\lambda X(t)}]=\exp\{-\tilde\gamma\,t\,\lambda^{\alpha}\}$ for some $\tilde\gamma\ge0$ and shows that with Lévy measure $\nu(dy)=y^{-1-\alpha}dy$ one has $\tilde\gamma=\Gamma(1-\alpha)/\alpha$; rescaling time yields the unit-scale form above. See also \cite[Chs.~1,\,3,\,5]{SSV12}.)

Fix $a>0$ and rename the Laplace parameter by $t\ge0$. Then
\[
e^{-a\,t^{\gamma}}
\;=\;
\mathbb{E}\!\left[e^{-\,t\,S_a}\right]
\;=\;
\int_{0}^{\infty} e^{-u t}\,\mathbb{P}(S_a\in du).
\]
By the Bernstein–Widder theorem (uniqueness of Laplace transforms of positive measures on $[0,\infty)$) (e.g. see \cite[Thm.~1.4]{SSV12}), we conclude \eqref{eq:bernstein-EN-eps} and, in addition, identifies $\nu_{\gamma,a}$ as the law of $S_a$.

In addition, for every \(\lambda>0\),
\begin{align*}
\mathbb{E}\big[e^{-\lambda S_a}\big]
&=e^{-a\,\lambda^\gamma}=\mathbb{E}\big[e^{-\lambda\,a^{1/\gamma}S_1}\big]
\end{align*}
By \emph{uniqueness of Laplace transforms on \([0,\infty)\)}, we conclude
\[
S_a\ \overset{d}{=}\ a^{1/\gamma} S_1.
\]

Hence, for any bounded measurable test function \(\varphi:[0,\infty)\to\mathbb{R}\),
\begin{align*}
\int_{[0,\infty)} \varphi(u)\,\nu_{\gamma,a}(du)
&=\mathbb{E}\big[\varphi(S_a)\big]
=\mathbb{E}\big[\varphi\!\big(a^{1/\gamma}S_1\big)\big]
=\int_{[0,\infty)} \varphi\!\big(a^{1/\gamma}s\big)\,\nu_{\gamma,1}(ds).
\end{align*}

In particular, \(\nu_{\gamma,a}(B)=\nu_{\gamma,1}\!\big(B/a^{1/\gamma}\big)\) for all Borel sets \(B\subset[0,\infty)\), which is precisely \eqref{eq:scale-EN-eps}.

Let \(\theta>0\). For every \(x>0\), we have (see the Gamma distribution):

\[
x^{-\theta}=\frac{1}{\Gamma(\theta)}\int_{0}^{\infty} s^{\theta-1} e^{-s x}\,ds.
\]

Apply this with \(x=S_1\) and use Tonelli’s theorem (the integrand is nonnegative) to justify the exchange of expectation and integral:
\begin{align*}
\int_{(0,\infty)} u^{-\theta}\,\nu_{\gamma,1}(du)&=\mathbb{E}\big[S_1^{-\theta}\big]
=\mathbb{E}\!\left[\frac{1}{\Gamma(\theta)}\int_{0}^{\infty} s^{\theta-1} e^{-s S_1}\,ds\right]
=\frac{1}{\Gamma(\theta)}\int_{0}^{\infty} s^{\theta-1}\,\mathbb{E}\big[e^{-s S_1}\big]\,ds\\
&=\frac{1}{\Gamma(\theta)}\int_{0}^{\infty} s^{\theta-1} e^{-s^\gamma}\,ds
\qquad\text{since }\mathbb{E}\big[e^{-s S_1}\big]=e^{-s^\gamma}.
\end{align*}
The integral is finite for all \(\theta>0\): near \(0\) the integrand behaves like \(s^{\theta-1}\), which is integrable on \((0,1]\), and for \(s\to\infty\) the factor \(e^{-s^\gamma}\) decays super-exponentially. 
With the change of variables \(v=s^\gamma\) (i.e., \(s=v^{1/\gamma}\), \(ds=(1/\gamma)\,v^{1/\gamma-1}dv\)), we obtain
\[
\int_{0}^{\infty} s^{\theta-1} e^{-s^\gamma}\,ds
=\frac{1}{\gamma}\int_{0}^{\infty} v^{\theta/\gamma-1} e^{-v}\,dv
=\frac{1}{\gamma}\,\Gamma\!\left(\frac{\theta}{\gamma}\right).
\]
Therefore,
\[
\int_{(0,\infty)} u^{-\theta}\,\nu_{\gamma,1}(du)=\mathbb{E}\big[S_1^{-\theta}\big]
=\frac{1}{\gamma}\,\frac{\Gamma(\theta/\gamma)}{\Gamma(\theta)}.
\]
which is precisely \eqref{eq:neg-moment-EN-eps}. 

\end{proof}

\begin{lemma}\label{lem:potter-EN-eps}
If $L$ is slowly varying, then for every $\varepsilon\in(0,1)$ and $C>1$ there exists $X(\varepsilon,C)\ge 1$ such that, for all $x\ge X(\varepsilon,C)$ and $y>0$,
\begin{equation}\label{eq:potter-EN-eps}
\frac{L(xy)}{L(x)}\ \le\ C\,\max\{y^\varepsilon,y^{-\varepsilon}\}.
\end{equation}
Moreover, for every compact interval $0<a\le y\le b<\infty$,
\begin{equation}\label{eq:UCT-EN-eps}
\lim_{x\to\infty}\ \sup_{y\in[a,b]}\ \left|\frac{L(xy)}{L(x)}-1\right|=0.
\end{equation}
\end{lemma}
\begin{proof}

Since $L$ is slowly varying at $\infty$, the Uniform Convergence Theorem
(see, e.g., \cite[Thm.~1.2.1]{BGT1987}) yields that for every compact
$K\subset(0,\infty)$,
\[
\sup_{\lambda\in K}\Bigg|\frac{L(\lambda x)}{L(x)}-1\Bigg|\xrightarrow[x\to\infty]{}0.
\]

We now prove \eqref{eq:potter-EN-eps}. We use Potter bounds for \emph{regularly varying} functions (see \cite[Thm.~1.5.6]{BGT1987}): for every $\varepsilon\in(0,1)$ and $C> 1$ there exists a constant \(X_0(\varepsilon,C)\) such that
\begin{equation}\label{eq:potter-RV}
\frac{L(xy)}{L(x)}\ \le C\,y^{\varepsilon}, \, y\ge 1, \qquad (x\ge X_0(\varepsilon,C)).
\end{equation}

Moreover, let \(g(x):=L(x)x^{\varepsilon}\); then \(g\) is regularly varying with index \(\varepsilon\), and the Uniform Convergence Theorem
(see, e.g., \cite[Thm.~1.5.2]{BGT1987}) yields that for every \(b>0\)
\[
\sup_{\lambda\in (0,b]}\Bigg|\frac{g(\lambda x)}{g(x)}-\lambda^{\varepsilon}\Bigg|\xrightarrow[x\to\infty]{}0.
\]

In particular, take \(b=1\). Now, given \(\delta>0\), there exists \(X_1(\delta)>0\) such that if \(x\geq X_1(\delta)\), then 
\[
\frac{L(xy)}{L(x)}=\frac{g(xy)}{g(x)y^{\varepsilon}}\leq \delta  y^{-\varepsilon}+1\leq(\delta+1)y^{-\varepsilon}
\]
for all \(y \in (0,1]\). Then, choosing \(\delta>0\) such that \(1+\delta\leq C\), we obtain that \eqref{eq:potter-EN-eps} holds.

\end{proof}

\begin{lemma}\label{lem:tail-tau1-EN-eps}
Let $(S_k)_{k\ge0}$ be the simple symmetric random walk on $\mathbb Z$ with $S_0=0$ and
\[
\tau_1:=\inf\{k\ge1:\ S_k=1\}.
\]
Then
\begin{equation}\label{tau1_t}
\PP(\tau_1>t)\ \sim\ \sqrt{\tfrac{2}{\pi}}\;t^{-1/2}\qquad (t\to\infty).
\end{equation}
\end{lemma}

\begin{proof}
By \cite[Prop.~5.1.2]{LawlerLimic2010} and by the spatial symmetry (translation by $1$ and reflection) of the simple symmetric random walk, we have that:
\[
\PP(\tau_1>2n)\ \sim\ \frac{1}{\sqrt{\pi\,n}}
\qquad (n\to\infty).
\]
Hence \eqref{tau1_t} follows.
\end{proof}

\begin{lemma}\label{lem:tauber-Tgamma-EN-eps}
Let $\gamma>1$. Then for every $\varepsilon\in(0,1)$ there exist $a_\varepsilon\in(0,1)$ and positive constants $c_1(\varepsilon)\le c_2(\varepsilon)$ such that, for all $a\in(0,a_\varepsilon]$,
\begin{equation}\label{eq:Lap-two-sided-EN-eps}
c_1(\varepsilon)\,a^{1/(2\gamma)}\ \le\ 1-\EE\!\big[e^{-a\,\tau_1^\gamma}\big]\ \le\ c_2(\varepsilon)\,a^{1/(2\gamma)}.
\end{equation}
Moreover,
\begin{equation}\label{eq:Lap-limit-EN-eps}
\lim_{a\downarrow 0}\ \frac{1-\EE(e^{-a \tau_1^\gamma})}{a^{1/(2\gamma)}}\ =\ \Gamma\!\Big(1-\frac{1}{2\gamma}\Big)\sqrt{\frac{2}{\pi}}.
\end{equation}
\end{lemma}

\begin{proof} Let \(\gamma>1\) and write \(\overline F_{\tau_1^\gamma}(y):=\PP(\tau_1^\gamma>y)\). Define \(\rho:=1/(2\gamma)\in(0,1/2)\) and \(c:=\sqrt{2/\pi}\). By Lemma~\ref{lem:tail-tau1-EN-eps}, \(\PP(\tau_1>t)\sim c\,t^{-1/2}\) as \(t\to\infty\). Hence
\begin{equation}\label{tailtau_1}
\overline F_{\tau_1^\gamma}(y)=\PP\!\big(\tau_1>y^{1/\gamma}\big)\sim c\,y^{-\rho}\qquad(y\to\infty),
\end{equation}
i.e., the tail of \(\tau_1^\gamma\) is regularly varying with index \(-\rho\) (standard terminology as in \cite[Ch.~1]{BGT1987}). For any nonnegative random variable,
\begin{equation}\label{TTL}
1-\EE\!\big(e^{-a\tau_1^\gamma}\big)\;=\;a\int_{0}^{\infty} e^{-ay}\,\overline F_{\tau_1^\gamma}(y)\,dy,\qquad a>0,
\end{equation}
see, e.g., the Laplace–Stieltjes integration-by-parts identity in equation (2.7) of \cite[Ch.~XIII]{Feller1971}.

Since \(F_{\tau_1^\gamma}(y)\) is nonincreasing, by (\ref{tailtau_1}) and Karamata’s Tauberian theorem for Laplace transforms applied to the regularly varying tail in (\ref{TTL}), we obtain
\[
1-\EE\!\big(e^{-a\tau_1^\gamma}\big)\ \sim\ c\,\Gamma(1-\rho)\,a^{\rho}\qquad(a\downarrow0),
\]
see \cite[Th.4,Ch. XIII.5]{Feller68}. With \(\rho=1/(2\gamma)\) and \(c=\sqrt{2/\pi}\), this gives the limit \eqref{eq:Lap-limit-EN-eps}.
In particular, for every \(\varepsilon\in(0,1)\) there exist \(a_\varepsilon\in(0,1)\) and constants \(c_1(\varepsilon),c_2(\varepsilon)>0\) such that, for all \(a\in(0,a_\varepsilon]\),
\[
c_1(\varepsilon)\,a^{1/(2\gamma)}\ \le\ 1-\EE\!\big(e^{-a\tau_1^\gamma}\big)\ \le\ c_2(\varepsilon)\,a^{1/(2\gamma)},
\]
which is \eqref{eq:Lap-two-sided-EN-eps}.
\end{proof}

\begin{lemma}\label{lem:convex-ineq-EN-eps}
For $x_i\ge 0$ and $\gamma> 1$,
\begin{equation}\label{eq:superadd}
\Big(\sum_{i=1}^n x_i\Big)^{\gamma}\ \ge\ \sum_{i=1}^n x_i^{\gamma}.
\end{equation}
Hence, if $T_i$ are i.i.d.\ nonnegative random variables and $a\ge0$,
\begin{equation}\label{eq:prod-bound-EN-eps}
\EE\!\big[e^{-a(\sum_{i=1}^n T_i)^\gamma}\big]
\ \le\ \EE\!\big[e^{-a\sum_{i=1}^n T_i^\gamma}\big]
\ =\ \prod_{i=1}^n \EE\!\big[e^{-a T_i^\gamma}\big].
\end{equation}
\end{lemma}

\begin{proof}We first prove \eqref{eq:superadd} for \(n=2\). Fix \(a\ge 0\) and, for \(b\ge 0\), set
\[
g_b(a):=(a+b)^\gamma-a^\gamma-b^\gamma .
\]
Since
\[
g_b'(a)=\gamma\big((a+b)^{\gamma-1}-a^{\gamma-1}\big)\ge 0
\]
(because \(\gamma\ge 1\) and the map \(a\mapsto a^{\gamma-1}\) is nondecreasing), it follows that \(g_b(a)\ge g_b(0)=b^\gamma-b^\gamma=0\). By symmetry, the same conclusion holds upon interchanging \(a\) and \(b\). An induction on \(n\) then yields \eqref{eq:superadd}.

Applying \eqref{eq:superadd} pointwise with \(x_i=T_i(\omega)\ge 0\) and multiplying by \(-a\le 0\),
\[
-a\Big(\sum_i T_i\Big)^\gamma \le -a\sum_i T_i^\gamma .
\]
Since the exponential function preserves order,
\[
e^{-a(\sum_i T_i)^\gamma}\ \le\ e^{-a\sum_i T_i^\gamma}\ =\ \prod_i e^{-a T_i^\gamma}.
\]
Taking expectations and using independence to factor the product, we obtain \eqref{eq:prod-bound-EN-eps}.

\end{proof}

\begin{lemma}\label{lem:inv-principle-EN-eps}
Let $(S_k)_{k\ge0}$ be the simple symmetric random walk on $\mathbb Z$ with $S_0=0$. For any $c_0>0$, there exists $n_0\in\mathbb N$ such that for all $n\ge n_0$,
\begin{equation}\label{eq:event-typical-EN-eps}
\PP\big(\tau_n< c_0\,n^2\big)\ \ge\ \theta,
\end{equation}
where \(\theta:=\ \tfrac12\big(1-\Phi(1/\sqrt{c_0})\big)\in(0,1)\) and $\Phi$ is the standard Gaussian cdf.
\end{lemma}

\begin{proof} Let $M_m:=\max_{0\le k\le m} S_k$. By definition of the hitting time,
\begin{equation}\label{MaxS_n}
\PP(\tau_n< m)=\PP(M_m> n)\ge \PP(S_m> n).
\end{equation}

\smallskip
Write $S_m=\sum_{i=1}^m X_i$, where $\{X_i\}_{i\in \mathbb{N}}$ is an i.i.d. collection of random variables with \ $X_i\in\{\pm1\}$, $\EE X_i=0$, $\text{Var}(X_i)=1$. By the Berry–Esseen theorem, for all $x\in\mathbb R$ and $m\in\mathbb N$,
\[
\Big|\PP\!\Big(\frac{S_m}{\sqrt m}\le x\Big)-\Phi(x)\Big|\ \le\ \frac{3}{\sqrt m},
\]
see, e.g., \cite[Thm.~3.4.17]{DurrettPTE}. 

Set $x_n:=n/\sqrt{m}$ with $m=\lfloor c_0 n^2\rfloor$. The uniform Berry–Esseen bound
(see Durrett, Thm.~3.4.17) gives, for all $x\in\mathbb R$,
\[
\sup_{x\in\mathbb R}\Big|\PP\!\Big(\frac{S_m}{\sqrt m}\le x\Big)-\Phi(x)\Big|
\;\le\;\frac{3}{\sqrt m}.
\]
Evaluating at $x=x_n$,
\[
\PP(S_m> n)
=1-\PP\!\Big(\frac{S_m}{\sqrt m}\le x_n\Big)
\ \ge\ 1-\Phi(x_n)-\frac{3}{\sqrt m}.
\]
Since $x_n\to 1/\sqrt{c_0}$ and $\Phi$ is globally Lipschitz with constant
\(1/\sqrt{2\pi}\), we have
\[
\big|\Phi(x_n)-\Phi(1/\sqrt{c_0})\big|\ \le\ \frac{|x_n-1/\sqrt{c_0}|}{\sqrt{2\pi}}.
\]
Moreover, writing $m=\lfloor c_0 n^2\rfloor=c_0 n^2-r_n$ with $r_n\in[0,1)$,
\[
x_n=\frac{n}{\sqrt{c_0 n^2-r_n}}
=\frac{1}{\sqrt{c_0}}\cdot\frac{1}{\sqrt{1-r_n/(c_0 n^2)}}
\le \frac{1}{\sqrt{c_0}}\Big(1+\frac{2\,r_n}{c_0 n^2}\Big)
\le \frac{1}{\sqrt{c_0}}+\frac{2}{c_0^{3/2}n^2},
\]
for all $n$ large enough (use $1/\sqrt{1-t}\le 1+2t$ for $t\in[0,1/2]$). Hence
\[
\big|\Phi(x_n)-\Phi(1/\sqrt{c_0})\big|\ \le\ \frac{1}{\sqrt{2\pi}}\cdot \frac{2}{c_0^{3/2}n^2}.
\]
Combining the last two displays,
\[
\PP(S_m> n)\ \ge\ \big(1-\Phi(1/\sqrt{c_0})\big)\ -\ \frac{3}{\sqrt m}\ -\
\frac{2}{\sqrt{2\pi}\,c_0^{3/2}n^2}.
\]
Let
\[
\theta\ :=\ \tfrac12\big(1-\Phi(1/\sqrt{c_0})\big)\in(0,1).
\]
Choose $n_0$ so large that, for all $n\ge n_0$, \(\frac{3}{\sqrt m}\ +\ \frac{2}{\sqrt{2\pi}\,c_0^{3/2}n^2}\ \le\ \theta\). Then, for all $n\ge n_0$, \(\PP(S_m> n)\ \ge\ \theta\). Since $\{S_m> n\}\subseteq\{\tau_n< m\}$, we conclude that
\[
\PP(\tau_n< c_0 n^2)\ \ge\ \PP(S_m> n)\ \ge\ \theta
\ =\ \tfrac12\big(1-\Phi(1/\sqrt{c_0})\big),
\]
which is \eqref{eq:event-typical-EN-eps}.

\end{proof}

\begin{acks}[Acknowledgments]
G.O.C. was supported by the Coordenação de Aperfeiçoamento de Pessoal de Nível Superior - Brasil (CAPES) - Finance Code 001. F.P.M. and J.H.R.G. are  supported by FAPESP (grants 2023/13453-5 and 2025/03804-0, respectively).
\end{acks}


\end{document}